\documentclass[12pt, reqno]{amsart}
 \usepackage{tikz}
 \usepackage{hyperref, mathtools,mathrsfs}
 \usepackage[skip=4pt]{caption}
 \usepackage{amssymb,amsmath,graphicx,amsthm}
  \usetikzlibrary{positioning}
  \usepackage{tkz-euclide}
\usepackage{rotating}
\usepackage{color,xcolor}
\usepackage{geometry}
\definecolor{darkred}{rgb}{1,0,0} 
\definecolor{darkgreen}{rgb}{0,0.8,0}
\definecolor{darkblue}{rgb}{0,0,1}

\hypersetup{colorlinks,
linkcolor=darkblue,
filecolor=darkgreen,
urlcolor=darkred,
citecolor=darkgreen}

 \def\bt{\begin{theorem}}
 	\def\el{\end{lemma}}
 \def\bl{\begin{lemma}}
 	\def\et{\end{theorem}}
 \def\bp{\begin{proposition}}
 	\def\ep{\end{proposition}}
 \def\bd{\begin{definition}}
 	\def\ed{\end{definition}}
 \def\br{\begin{remark}}
 	\def\er{\end{remark}}

 \newcommand{\e}[1]{\bold{e_{#1}}}

 \def\R{{\mathbb R}}

\def\S{{\mathbb S}}
\def\S1{{\mathbb S^1}}

\def\e{{\bold e}}

 \numberwithin{equation}{section}

 \theoremstyle{plain}
 \newtheorem*{theorem*}{Theorem}

 \newtheorem{theorem}{Theorem}[section]

 \newtheorem{lemma}[theorem]{Lemma}
 
 \newtheorem{proposition}[theorem]{Proposition}

 \theoremstyle{definition}
 
 \newtheorem{definition}[theorem]{Definition}
 
 \theoremstyle{remark}

 \newtheorem{remark}[theorem]{Remark}

 \newcommand{\p}{\partial}

 \theoremstyle{plain} 

\begin{document}
 	
 	\title{Totally integrable symplectic billiards are ellipses}  
\author{Luca Baracco, Olga Bernardi}
\address{Dipartimento di Matematica Tullio Levi-Civita, Universit\`a di Padova, via Trieste 63, 35121 Padova, Italy}
\email{baracco@math.unipd.it, obern@math.unipd.it}

 
\maketitle

\begin{abstract}

In this paper we prove that a totally integrable strictly-convex symplectic billiard table, whose boundary has everywhere strictly positive curvature, must be an ellipse. The proof, inspired by the analogous result of Bialy for Birkhoff billiards, uses the affine equivariance of the symplectic billiard map.
\newline\noindent
\subjclass[MSC2020]{Primary 37E40. Secondary 37J35 }
\end{abstract}

\section{Introduction}
\noindent Symplectic billiards were introduced by P. Albers and S. Tabachnikov \cite{AT} in 2018 as a simple dynamical system where --opposed to Birkhoff billiards-- the generating function is the area form instead of the length. In a planar strictly-convex region with smooth boundary, this variational formulation gives rise to the ``symplectic billiard map'' (see Figure \ref{Figuraccia}): if the points $x$ and $z$ are fixed, the position of the intermediate point $y$ on the boundary is determined by the condition that the tangent line at $y$ is parallel to the segment joining $x$ and $z$. As showed in \cite{AT}[Section 2], the symplectic billiard map results a monotone, twist map, preserving an area form. Moreover, it commutes with affine transformations of the plane. \\
\indent In general, a crucial dynamical behavior for billiards is the so-called integrability. We say that a billiard is \textit{integrable} if there exists a regular (at least continuous) foliation of the phase-space consisting of invariant curves for the billiard map. According to the notion of integrability one considers, the foliation is required to be of full measure and/or the corresponding invariant curves closed. Moreover, a billiard is called \textit{totally integrable} if the phase-space is fully foliated by continuous invariant curves which are not null-homotopic. In literature, \textit{full global integrability} and \textit{$C^0$ integrability} are used as synonyms of total integrability, see e.g. \cite{KaSo}[End of Page 8] and \cite{Arnaud}[Page 884] (in the Hamiltonian setting). \\
\indent We stress that integrability remains nowadays an unanswered property also for Birkhoff billiards. In fact, the celebrated Birkhoff conjecture (first appeared in \cite{Bir} and \cite{Por}), which says that the unique integrable Birkhoff billiards are circles and ellipses, is still open. It is worth noting that the so-called perturbative and local Birkhoff conjectures have been successfully studied, see \cite{AvAltri}, \cite{KaSo} and \cite{HKS} respectively. Moreover, M. Bialy and A.E. Mironov \cite{BiMi} recently proved the Birkhoff-Poritsky conjecture for centrally-symmetric $C^2$-smooth convex planar billiards. We refer to \cite{KS} for a concise and self-contained discussion of integrability for Birkhoff billiards. \\
\indent Concerning the total integrability, in 1992 Bialy proved that {\textit{totally integrable Birkhoff billiards are necessarily circles}}. In order to prove his result, Bialy has been inspired by a method by E. Hopf for Riemannian metrics without conjugate points. In particular, he formulated a discrete version of this argument for positive twist maps of the open strip $\mathbb{T} \times (-1,1)$, preserving the Lebesgue measure. In the case of Birkhoff billiards, Bialy's construction of a non-vanishing Jacobi field along each billiard configuration --combined with the planar isoperimetric inequality-- gives the result. An alternative approach to such a theorem was proposed some years later by M.P. Wojtkowski, see \cite{Wo}. This new proof of Bialy's result is based on dynamical --instead of variational-- arguments, like the so-called ``mirror equation'' from geometric optics.  \\
\indent We shall assume that the boundary of the billiard table has everywhere strictly positive curvature. The aim of this paper is to prove that {\textit{totally integrable symplectic billiards are necessarily ellipses}}. In other words:
\begin{theorem} \label{bialy}
If the phase-space of the symplectic billiard map is foliated by continuous invariant closed curves not null-homotopic then the billiard table is an ellipse.
\end{theorem}
\noindent The proof is a (non-trivial) adaptation to the symplectic billiard case of Bialy's arguments. As in his case, the center point is showing this result.
\begin{theorem} \label{luca}
The only symplectic billiards without conjugate points are elliptic billiards.
\end{theorem}
\noindent Differently from Birkhoff case, the above theorem is obtained by using a precise  coordinate system chosen among a class of affine transformations. 
The key point of the proof is Proposition \ref{P1} where we show that in this new coordinates the second Fourier terms of the boundary are $0$. 
The conclusion follows by using the estimate obtained through the technique of \cite{B} together with the planar isoperimetric inequality. Finally, by nowadays ``standard'' variational arguments for twist maps (as in \cite{B}[Page 154]), we have that Theorem \ref{bialy} follows from Theorem \ref{luca}. \\
\indent We remark that changing coordinates as we did is a natural technique to obtain refined isoperimetric estimates. It came to our attention that in the paper \cite{WE} the authors perform a family of affine transformations in order to cancel the first Fourier terms of the boundary. However, our approach is completely different essentially for two reasons. First, the vanishing of the second Fourier terms is aimed to prove the converse of the planar isoperimetric inequality; second, our argument is more topological rather than analytical.  \\  
\indent The paper is organized as follows. In Section \ref{Se2} we introduce symplectic billiards and recall their main properties. In Section \ref{BiIn} we present an inequality for billiards without conjugate points --see essentially \cite{B}[Theorem 1]-- which is the main ingredient in the proof of Theorem \ref{luca}. Section \ref{Se3} is devoted to technical facts on integrals involving the area form. 
In Section \ref{principale} we show how to choose the right coordinate system and the parametrization for $\partial D$ in order to get the strongest estimate which, together with the planar isoperimetric inequality, will yield the conclusion. In this section, which represents the core of the paper, inequality (\ref{eccola2}) is treated in a completely different way with respect to the Birkhoff case in order to prove Theorem \ref{luca}.
\begin{figure}
    \centering
    \includegraphics{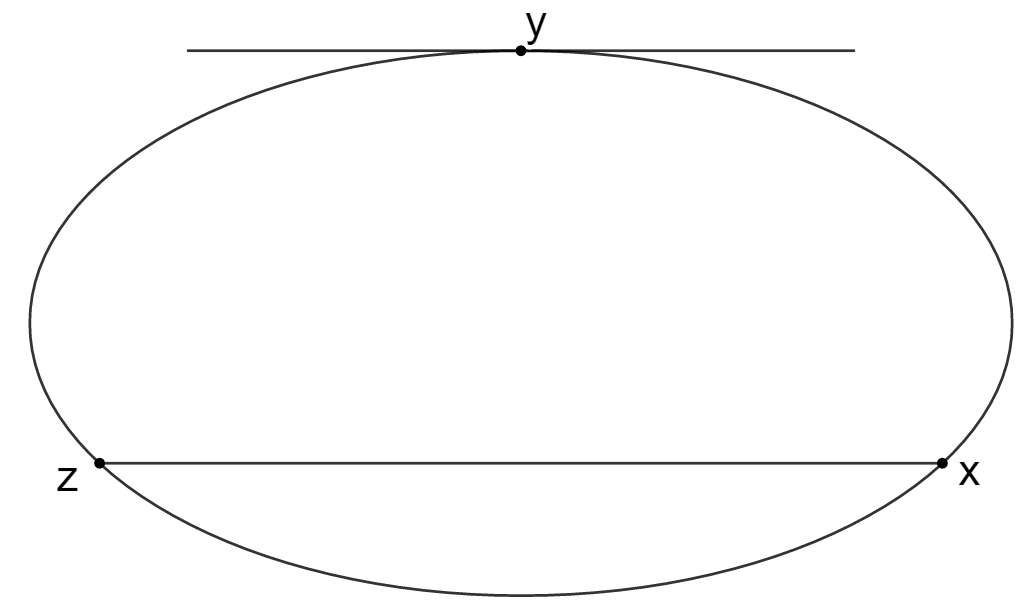}
    \caption{The symplectic billiard map reflection.}
    \label{Figuraccia}
\end{figure}

\section{The dynamical system} \label{Se2}
\noindent Let $D$ be a strictly-convex domain in $(\R^2,\omega)$ with smooth boundary $\partial D$. Assume that $D$ contains the origin and  that the perimeter of $\partial D$ is normalized to one. Fixed the positive counter-clockwise orientation, let $\gamma :\mathbb{T} \rightarrow \p D$ be a smooth parametrization of $\p D$, where $\mathbb{T} = \mathbb{R}/ \ell \, \mathbb{Z} = [0,\ell] / \sim$ identifying $0 \sim \ell$. Since $D$ is  strictly-convex, for every point $\gamma(t) \in \p D$ there exists a unique point $\gamma(t^*)$ such that 
$$\omega(\gamma'(t),\gamma'(t^*)) = 0.$$  
(In the above formula, $\gamma'(t)$ and $\gamma'(t^*)$ denote the tangent vectors at $\p D$ in $\gamma(t)$ and $\gamma(t^*)$ respectively). We refer to
\begin{equation} \label{PS}
\mathcal{P} := \{ (\gamma(t), \gamma(s)) \in \p D \times \p D: \ \gamma(t) < \gamma(s) < \gamma(t^*)\}
\end{equation}
as the (open, positive) phase-space and we define the symplectic billiard map as follows (see \cite{AT} [Page 5]).
$$\Phi: \mathcal{P} \rightarrow \mathcal{P}, \qquad (\gamma(t_1),\gamma(t_2) \mapsto (\gamma(t_2),\gamma (t_3))$$
where $\gamma(t_3)$ is the unique point satisfying 
$$\omega( \gamma^{\prime} (t_2),\gamma(t_3)-\gamma(t_1))=0.$$ 
Here are some properties of the symplectic billiard map (we refer to \cite{AT} [Section 2] for exhaustive details). \\
~\newline
\noindent 
-- The symplectic billiard map commutes with affine transformations of the plane, since they preserve tangent directions. Clearly, the symplectic billiard and the Birkhoff billiard maps coincides for circles, which result totally integrable. By the affine equivariance of the symplectic billiard map and differently to the Birkhoff case, also ellipses are totally integrable symplectic billiards. In view of our theorem, these are the unique cases. \\
~\newline
-- $\Phi$ is continuous and can be continuously extended to $\bar{\mathcal{P}}$ so that 
$$\Phi(\gamma(t),\gamma(t)) = (\gamma(t),\gamma(t)) \qquad \text{and} \qquad \Phi(\gamma(t),\gamma(t^*)) = (\gamma(t^*),\gamma(t)).$$
\noindent -- The (standard) area form
$$\omega: \mathcal{P} \to \mathbb{R}, \qquad (\gamma(t_1),\gamma(t_2)) \mapsto \omega(\gamma(t_1),\gamma(t_2))$$
is a generating function for $\Phi$, that is
\begin{equation} \label{GF}
\Phi(\gamma(t_1),\gamma(t_2)) = (\gamma(t_2),\gamma(t_3)) \qquad \Longleftrightarrow \qquad \omega(\gamma(t_1),\gamma'(t_2)) + \omega(\gamma'(t_2),\gamma(t_3)) = 0.
\end{equation}
-- Let introduce new variables
$$s_1 = \omega(\gamma'(t_1), \gamma(t_2)), \qquad s_2 = \omega(\gamma(t_1), \gamma'(t_2))$$
and the $2$-form 
\begin{equation} \label{2 forma}
\Omega = \omega(\gamma'(t_1),\gamma'(t_2)) dt_1dt_2 = \frac{\partial s_1}{\partial t_2}(t_1,t_2) dt_1dt_2.
\end{equation}
Then $(t,s)$ are coordinates on $\mathcal{P}$ and $\Omega$ is an area form. Moreover, the symplectic billiard map becomes a (positive) twist map ($\frac{\partial t_2}{\partial s_1} > 0$)  which preserves $\Omega$. 

\section{An inequality by M. Bialy} \label{BiIn}
\noindent From now on, $L(t_1,t_2) := \omega(\gamma(t_1),\gamma(t_2))$ denotes the generating function for $\Phi$ and $L_{ij}$ (for $i,j=1,2$) are the usual partial derivatives:
$$L_{11}(t_1,t_2) = \omega(\gamma''(t_1),\gamma(t_2)), \quad L_{22}(t_1,t_2) = \omega(\gamma(t_1),\gamma''(t_2)) \quad \text{and} \quad  L_{12}(t_1,t_2) = \omega(\gamma'(t_1),\gamma'(t_2)).$$
Therefore,
\begin{equation} \label{estimate}
|L_{11}|, |L_{22}|, L_{12} \le K
\end{equation}
where $K$ is a constant depending on the $C^2$-norm of $\gamma$. \\
\noindent Let $\{t_n\}_{n \in \mathbb{Z}}$ be a symplectic billiard configuration. Then, by (\ref{GF}), $\{t_n\}_{n \in \mathbb{Z}}$ satisfies
$$L_1(t_n,t_{n+1}) +L_2(t_{n-1},t_n) = 0 \qquad \forall n \in \mathbb{Z}.$$
\noindent In the sequel, we recall the definitions of (discrete) Jacobi field and conjugate point introduced by M. Bialy, see \cite{B} [Definitions 1 and 2].  Using the parametrization of $\p D$ we identify the tangent bundle $T(\p D)$ with $\p D \times \R$.
\begin{definition} $(i)$ A sequence $\{\xi_n\}_{n \in \mathbb{Z}}$, $\xi_n \in T_{\gamma(t_n)}\partial D$ is called a Jacobi field along $\{t_n\}_{n \in \mathbb{Z}}$ \\
if it satisfies:
$$ L_{12}(t_{n-1}, t_n) \xi_{n-1} + [L_{22}(t_{n-1}, t_n) + L_{11}(t_{n}, t_{n+1})] \xi_n + L_{12}(t_{n}, t_{n+1}) \xi_{n+1} = 0
$$
for all $n \in \mathbb{Z}$. \\
~\newline
\noindent 
$(ii)$ Let $M < N$. Two points $t_M$, $t_N$ of a symplectic billiard configuration are called conjugate if there is a non-zero Jacobi field vanishing at $t_M$ and $t_N$. \\
~\newline
$(iii)$ A symplectic billiard is without conjugate points if every billiard configuration has no conjugate points.
\end{definition}
\noindent Suppose now that the symplectic billiard is without conjugate points. Since the corresponding map is (positive) twist and preserves an area form, the arguments of \cite{B}[Section 3] can be rephrased step by step in order to prove the next proposition.
\begin{proposition} \label{DISUG}
If the symplectic billiard has no conjugate points, then the following inequality holds:
\begin{equation} \label{eccola2}
\int_{\mathcal{P}} L_{11}(t_1,t_2) + 2L_{12}(t_1,t_2) + L_{22}(t_1,t_2)\, d\Omega \le 0.
\end{equation}
\end{proposition}
\begin{proof} We only sketch the proof, we refer to \cite{B}[Section 3] for all details. First, by using the (positive) twist condition and the absence of conjugate points assumption, it is possible to construct a proportional, strictly positive Jacobi field $\{\nu_n\}_{n \in \mathbb{Z}}$ and a corresponding well-defined function on $\mathcal{P}$:
$$F(\gamma(t_n),\gamma(t_{n+1})) := -L_{11}(t_n,t_{n+1}) - L_{12}(t_n,t_{n+1}) \frac{\nu_{n+1}}{\nu_n}.$$
As a consequence of estimates (\ref{estimate}), $F \in L_1(\mathcal{P},\Omega)$. Then, by the (positive) twist condition again, it can be easily proved --exactly as in \cite{B}[Lemma 4]-- that the function $F: \mathcal{P} \to \mathbb{R}$ satisfies the following inequality:
$$F(\Phi(\gamma(t_1),\gamma(t_2)) - F(\gamma(t_1),\gamma(t_2)) \ge L_{11}(t_1,t_2) + 2L_{12}(t_1,t_2) + L_{22}(t_1,t_2).$$
Integrating both sides over $(\mathcal{P},\Omega)$ and using the fact that the measure $\Omega$ is invariant under $\Phi$, we obtain the desired inequality.
\end{proof}
\noindent The next section is devoted to study accurately the integral in the above proposition, that is (by (\ref{2 forma})):
\begin{equation} \label{eccola}
\int_{\mathcal{P}} [L_{11}(t_1,t_2) + 2L_{12}(t_1,t_2) + L_{22}(t_1,t_2)] L_{12}(t_1,t_2) \, dt_1 dt_2.
\end{equation}

\section{Some basic computations} \label{Se3}
\noindent We start this section by proving some straightforward facts, which will be useful in the sequel.  
\begin{lemma} For $L(t_1,t_2) = \omega(\gamma(t_1),\gamma(t_2))$, the next equality holds:
\[
\begin{split}&2 \int_{\mathcal{P}} [L_{11}(t_1,t_2) + 2L_{12}(t_1,t_2) + L_{22}(t_1,t_2)] L_{12}(t_1,t_2) \, dt_1 dt_2  \\
&= \iint_{[0,\ell]^2} [L_{11}(t_1,t_2) + 2L_{12}(t_1,t_2) + L_{22}(t_1,t_2)] L_{12}(t_1,t_2) \, dt_1 dt_2.
\end{split}
\]
\end{lemma}
\begin{proof} Let consider the phase-space $\mathcal{P}$ --see (\ref{PS})-- and define
$$\widetilde{\mathcal{P}} := \{ (\gamma(s), \gamma(t)): \  (\gamma(t), \gamma(s)) \in \mathcal{P}\}.$$
Since $L_{11}(t_1,t_2) = -L_{22}(t_2,t_1)$ and $L_{12}(t_1,t_2) = -L_{12}(t_2,t_1)$, we have that
\[
\begin{split}&\int_{\mathcal{P}} [L_{11}(t_1,t_2) + 2L_{12}(t_1,t_2) + L_{22}(t_1,t_2)] L_{12}(t_1,t_2) \, dt_1 dt_2 \\ 
&= \int_{\mathcal{P}} [L_{11}(t_2,t_1) + 2L_{12}(t_2,t_1) + L_{22}(t_2,t_1)] L_{12}(t_2,t_1) \, dt_1 dt_2 \\
&= \int_{\widetilde{\mathcal{P}}} [L_{11}(t_1,t_2) + 2L_{12}(t_1,t_2) + L_{22}(t_1,t_2)] L_{12}(t_1,t_2) \, dt_1 dt_2
\end{split}
\]
By the $\ell$-periodicity of the generating function $L(t_1,t_2)$ with respect to both arguments, we obtain the desired equality. 
\end{proof}
 
\begin{lemma} \label{AREA} For $L(t_1,t_2) = \omega(\gamma(t_1),\gamma(t_2))$, the next equality holds:
\begin{equation} \label{2}
\iint_{[0,\ell]^2}\left[L_{11}(t_1,t_2)+L_{22}(t_1,t_2)\right] L_{12}(t_1,t_2) \, dt_1 dt_2= - 2A(D)\int_0^{\ell} \omega(\gamma''(t), \gamma'(t))\, dt,
\end{equation}
where $A(D)$ denotes the area of the billiard table $D$. 
\end{lemma}
\begin{proof} The proof is a simple integration by parts argument. In fact, we have
\[
\begin{split}
&\iint_{[0,\ell]^2}L_{11}(t_1,t_2)L_{12}(t_1,t_2) \, dt_1 dt_2 =\iint_{[0,\ell]^2} \omega(\gamma''(t_1),\gamma(t_2)) \omega(\gamma'(t_1), \gamma'(t_2)) \, dt_1 dt_2 \\
&=\iint_{[0,\ell]^2} \left( \gamma_1''(t_1) \gamma_2(t_2) - \gamma_2''(t_1) \gamma_1(t_2) \right) \left( \gamma'_1(t_1) \gamma'_2(t_2)-\gamma'_2(t_1) \gamma'_1(t_2) \right) \, dt_1dt_2 \\
&=\iint_{[0,\ell]^2} \left[-\gamma_1''(t_1) \gamma_2(t_2)\gamma'_2(t_1) \gamma'_1(t_2)-\gamma_2''(t_1) \gamma_1(t_2)\gamma'_1(t_1) \gamma'_2(t_2)\right] \, dt_1dt_2 \\
&=\iint_{[0,\ell]^2} \left[\gamma_1''(t_1) \gamma'_2(t_2)\gamma'_2(t_1) \gamma_1(t_2)-\gamma_2''(t_1) \gamma_1(t_2)\gamma'_1(t_1) \gamma'_2(t_2)\right] \, dt_1dt_2 \\ & = -A(D)\int_0^{\ell} \omega(\gamma''(t), \gamma'(t)) \, dt.
\end{split}
\]
By repeating the same computations on the term $L_{22}$ we have the conclusion.
\end{proof}
\noindent
We now parametrize $\p D$ by the direction of its tangent line. Let denote by 
$$\e_\alpha=(-\sin(\alpha),\cos(\alpha))$$ 
the unit vector which forms an angle $\alpha \in [0,2\pi]$ with respect to the fixed direction $(0,1)$. Since $D$ is strictly convex and contains the origin, for every $\alpha$ there exists a unique point $\gamma(t_{\alpha}) \in \p D$  such that $\gamma'(t_\alpha)=\|\gamma'(t_\alpha)\|\e_{\alpha} $. Let $p:[0,2\pi]\rightarrow \R_+$ 
be the corresponding support function, that is the distance from the origin of the tangent line to $\p D$ at $\gamma(t_\alpha)$. It is easy to see that the following relation holds:
\begin{equation} \label{1}
    \gamma(t_\alpha)=p' (\alpha)\e_\alpha - p(\alpha) J\e_\alpha 
\end{equation}
where $J$ is the rotation of angle $\frac{\pi}{2}$ in the positive verse. In the sequel, for the sake of simplicity, we write $\gamma (\alpha)$ instead of $\gamma(t_\alpha)$. With this parametrization, the perimeter length of $\partial D$ and the area of $D$ are given respectively by the integrals:
$$\int^{2 \pi}_0 p(\alpha) \, d\alpha = \int^{2 \pi}_0 (p''(\alpha) + p(\alpha)) \, d\alpha \quad \text{and} \quad \frac{1}{2} \int_0^{2 \pi} [p(\alpha) + p''(\alpha)] p(\alpha) \, d\alpha.$$
see for example \cite{Fla}[Section 2]. We note that, as a consequence of the positive curvature assumption of $\partial D$, it holds that $$(p''(\alpha) + p(\alpha))>0.$$
Therefore, the twist condition is also satisfied for $L(\alpha_1,\alpha_2) = \omega (\gamma(\alpha_1),\gamma(\alpha_2))$, indeed:
$$L_{12}(\alpha_1,\alpha_2) = (p''(\alpha_1) + p(\alpha_1))(p''(\alpha_2) + p(\alpha_2)) \sin(\alpha_2-\alpha_1) > 0$$
on $\mathcal{P} = \{(\alpha_1,\alpha_2): \ 0 < \alpha_2 - \alpha_1 < \pi \}$.
\begin{lemma}
    For $L(\alpha_1,\alpha_2) = \omega (\gamma(\alpha_1),\gamma(\alpha_2))$, the next equalities hold:
 \begin{equation} \label{tre}
 \iint_{[0,2\pi]^2} [L_{11}(\alpha_1,\alpha_2) + L_{22}(\alpha_1,\alpha_2)] L_{12}(\alpha_1,\alpha_2) \, d\alpha_1d\alpha_2 = -2A(D) \int_0^{2\pi} (p'' + p)^2 \, d\alpha
 \end{equation}
 and
    \begin{equation}\label{3}
    \begin{split}
        &2\iint_{[0,2\pi]^2} L_{12}^2(\alpha_1,\alpha_2)\, d\alpha_1d\alpha_2 \\
        &=\left( \int^{2\pi}_0 (p''+p)^2\, d\alpha \right)^2-\left( \int^{2\pi}_0 (p''+p)^2\cos(2\alpha)\, d\alpha \right)^2 -\left( \int^{2\pi}_0 (p''+p)^2\sin(2\alpha)\, d\alpha \right)^2.
    \end{split}
    \end{equation}
\end{lemma}
\begin{proof}
By starting from \eqref{1}, it is immediate to obtain that
$$\gamma'(\alpha) = (p''(\alpha) + p(\alpha)) \bold{e}_{\alpha}$$
and
$$\gamma''(\alpha) = (p'''(\alpha) + p'(\alpha)) \bold{e}_{\alpha} + (p''(\alpha) + p(\alpha)) J\bold{e}_{\alpha}.$$
As a consequence, 
$$\omega(\gamma''(\alpha),\gamma'(\alpha)) = (p''(\alpha) + p(\alpha))^2$$
and equality (\ref{tre}) is an immediate consequence of Lemma \ref{AREA}. \\
In order to prove equality (\ref{3}), we observe that the integrand 
$$ L_{12}^2(\alpha_1,\alpha_2)=(p''(\alpha_1)+p(\alpha_1))^2(p''(\alpha_2)+p(\alpha_2))^2\sin^2(\alpha_1-\alpha_2).$$
Since
$$\sin^2(\alpha_1-\alpha_2)= \frac{1}{2}\left( 1-\cos(2\alpha_1)\cos(2\alpha_2)-\sin(2\alpha_1)\sin(2\alpha_2)\right),$$
by taking the double integral we have the conclusion.
\end{proof}

\section{Proof of the main Theorem} \label{principale}

\noindent This section is devoted to prove Theorems \ref{bialy} and \ref{luca}. \\
~\newline
\indent We first study how integrals (\ref{tre}) and (\ref{3}) transform after a certain class of affine transformations. Specifically, coordinates changes given by a rotation around the origin followed by a unitary, diagonal, linear transformation. Clearly, a rotation around the origin of an angle $\sigma \in [0, 2\pi]$ doesn't substantially change \eqref{3} and \eqref{tre}, since the only effect in the integrals is that $p(\alpha)$ is replaced by $p(\alpha+\sigma)$. For $a>0$ let 
$$\varphi_a (x,y)=\left(ax,\frac ya\right)$$ 
be a unitary affine transformation. We denote $D_a=\varphi_a(D)$. It is easy to see that if $p$ is the support function of $D$, then the support function $p_a$ of $D_a$ is given by $$p_a(\psi)=p(\alpha(\psi))\left(\frac 1{a^2}\sin^2(\psi)+a^2\cos^2(\psi)\right)^{\frac{1}{2}}$$ 
where
\begin{equation}\label{10}
 \alpha(\psi)=\int_0^\psi \frac{1}{a^2\cos^2(t)+\frac 1{a^2}\sin^2(t)} dt .
 \end{equation}
In particular, if we apply $\varphi_a$ after a rotation of an angle $\sigma\in [0,2\pi]$, and call the resulting transformation $\varphi_{a,\sigma}$, then the support function of $D_{a,\sigma}:= \varphi_{a,\sigma}(D)$ is
\begin{equation}\label{8}
p_{a,\sigma}(\psi)=p(\alpha(\psi)+\sigma)\left(\frac{1}{a^2}\sin^2(\psi)+a^2\cos^2(\psi)\right)^{\frac{1}{2}} .
\end{equation}
The next result is on the basis of the proof of Theorem \ref{luca}. 
\begin{proposition}\label{P1} There exists a couple $(a,\sigma)\in \R_+\times[0,\frac{\pi}{2}]$ such that
\begin{equation}\label{9}
    \int^{2\pi}_0 p_{a,\sigma}(\psi)\cos(2\psi)\,d\psi=\int^{2\pi}_0 p_{a,\sigma}(\psi)\sin(2\psi)\,d\psi=0.
\end{equation}
\end{proposition}
\begin{proof} By plugging \eqref{8} into the integrals in \eqref{9}, we first make a change of variable in order to get a simpler argument inside $p$. To this end we take the inverse of \eqref{10} and define
 \begin{equation}
  \mathcal{I}_1(\sigma,a):=\int^{2\pi}_0 p_{a,\sigma}(\psi)\cos(2\psi)\,d\psi=
  \int^{2\pi}_0 \frac{p(\alpha+\sigma)(\frac{1}{a^2}\cos^2(\alpha)-a^2\sin^2(\alpha)) }{(\frac{1}{a^2}\cos^2(\alpha)+a^2\sin^2(\alpha))^{\frac{5}{2}}}\, d\alpha
 \end{equation}
 and
 \begin{equation}
  \mathcal{I}_2(\sigma,a):=   \int^{2\pi}_0 p_{a,\sigma}(\psi)\sin(2\psi)\,d\psi=
  \int^{2\pi}_0 \frac{p(\alpha+\sigma)2\sin(\alpha)\cos(\alpha) }{(\frac{1}{a^2}\cos^2(\alpha)+a^2\sin^2(\alpha))^{\frac{5}{2}}}\, d\alpha.
 \end{equation}
 It is easy to check that both the integrals depend continuously on the parameters $a \in \R_+$ and $\sigma \in [0,2\pi]$ and moreover that
 \begin{equation}\label{14}
\mathcal{I}_i(\sigma+\pi,a)=\mathcal{I}_i(\sigma,a) \quad \text{and} \quad \mathcal{I}_i\left(\sigma+\frac\pi 2,\frac 1a\right)=-\mathcal{I}_i(\sigma,a)
\end{equation}
for $i=1,2$. \\
\indent In the sequel, we prove that 
 \begin{equation} \label{11}
     \lim_{a\to 0^+} a \, \mathcal{I}_1\left(\sigma,a\right) = c\left(p\left(\sigma +\frac{\pi}{2}\right)+p\left(\sigma+\frac{3\pi}{2}\right)\right)
 \end{equation}
 where $c:= 2\int^{+\infty}_{0} \frac{t^2-1}{(t^2+1)^{\frac{5}{2}}} dt < 0$.
We consider the function inside the integral $\mathcal{I}_1$, which can be written as:
 \begin{equation*} 
  \left(a^3p(\alpha+\sigma)\right)\left(\frac{\cos^2(\alpha)-a^4\sin^2(\alpha) }{(\cos^2(\alpha)+a^4\sin^2(\alpha))^{\frac{5}{2}}}\right).
 \end{equation*}
 Clearly, the above function goes to $0$ uniformly as $a\to 0^+$ for $\alpha$ outside a neighborhood of $\frac{\pi}{2}$ and $\frac{3\pi}{2}$. \\
 In order to study the integral near $\frac{\pi}{2}$ (the argument near $\frac{3\pi}{2}$ will be the same), we replace $p$ with the first two terms from Taylor's formula and get
$$p (\alpha+\sigma) = p\left(\sigma + \frac{\pi}{2} \right) 
+ p' \left(\sigma + \frac{\pi}{2} \right) \left(\alpha-\frac{\pi}{2}\right)
+ O  \left(\alpha-\frac{\pi}{2}\right) ^2 $$ 
We consequently break locally the integral in three pieces. The constant term yields
\begin{equation} \label{limite}
\begin{split}
    \int^{\frac{\pi}{2} +\delta}_{{\frac{\pi}{2} -\delta}} a^3\left(\frac{\cos^2(\alpha)-a^4\sin^2(\alpha) }{(\cos^2(\alpha)+a^4\sin^2(\alpha))^{\frac{5}{2}}}\right) \, d\alpha &=  \int^{\frac{\pi}{2} +\delta}_{{\frac{\pi}{2} -\delta}} a^3\left(\frac{(1+a^4)\cos^2(\alpha)-a^4 }{((1-a^4)\cos^2(\alpha)+a^4)^{\frac{5}{2}}}\right) \, d\alpha \\ 
    &= 2\int^{\sin(\delta)}_0 a^3 \frac{(1+a^4)s^2 -a^4}{((1-a^4)s^2 +a^4)^{\frac{5}{2}}\sqrt{1-s^2}} \, ds \\
    &= \frac{2}{a}\int^{\frac{\sin(\delta)}{a^2}}_{0} \frac{(1+a^2)x^2-1}{((1-a^4)x^2+1)^{\frac{5}{2}}\sqrt{1-x^2a^4}} \,dx \\ 
    &\simeq \frac{2}{a}\int^{+\infty}_0  \frac{t^2-1}{(t^2+1)^{\frac{5}{2}}} \, dt
    \end{split}
\end{equation}
\noindent where $\simeq$ means that the two quantities are equal for $a \to 0^+$ up to a negligible term. The term in the middle: 
$$ 
\int^{\frac{\pi}{2} +\delta}_{{\frac{\pi}{2} -\delta}} p'(\sigma+\frac{\pi}{2})(\alpha-\frac{\pi}{2})\left(\frac{\cos^2(\alpha)-a^4\sin^2(\alpha) }{(\cos^2(\alpha)+a^4\sin^2(\alpha))^{\frac{5}{2}}}\right)\, d\alpha = 0
$$
for symmetry reasons around $\frac{\pi}{2}$. Finally, by arguments similar to (\ref{limite}), every function in $O  \left(\alpha-\frac{\pi}{2}\right) ^2$ gives a term estimated by $\frac{\delta^2}{a}$, hence negligible with respect to (\ref{limite}). Similar considerations hold for $\alpha=\frac{3\pi}{2}$ and \eqref{11} follows. \\
\indent We proceed to prove this limit for $\mathcal{I}_2$:
\begin{equation}\label{12}
    \lim_{a\to 0^+} \frac{\mathcal{I}_2(\sigma,a)}a=d\left(p'(\sigma+\frac{\pi}{2})+p'(\sigma+\frac{3\pi}{2})\right)  
\end{equation}
where $d:=-4\int^{+\infty}_0 \frac{t^2}{(1+t^2)^{\frac{5}{2}}} \, dt$. In fact, the integrand function in $\mathcal{I}_2$ is
 \begin{equation*}
  \left(a^5p(\alpha+\sigma)\right)\left(\frac{\sin(2\alpha)}{(\cos^2(\alpha)+a^4\sin^2(\alpha))^{\frac{5}{2}}}\right)   
 \end{equation*}
 which goes uniformly to $0$ as $a\to 0^+$ outside a neighborhood of $\frac{\pi}{2}$ and $\frac{3\pi}{2}$. By replacing $p$ with the first two terms of the Taylor formula we see that --in such a case-- the constant term does not give a contribution while the term of degree one near $\frac{\pi}{2}$ gives:
 
 \begin{equation}\label{secondo limite}
 \begin{split}
     \int^{\frac{\pi}{2} +\delta}_{{\frac{\pi}{2} -\delta}} a^5\left(\frac{2\sin(\alpha)\cos(\alpha) (\alpha-\frac{\pi}{2})}{(\cos^2(\alpha)+a^4\sin^2(\alpha))^{\frac{5}{2}}}\right)\, d\alpha&\simeq  \int^{\frac{\pi}{2} +\delta}_{{\frac{\pi}{2} -\delta}} a^5\left(\frac{-2(\alpha-\frac{\pi}{2})^2 }{((\alpha-\frac{\pi}{2})^2+a^4)^{\frac{5}{2}}}\right)\, d\alpha \\ &\simeq -4a\int^{\frac\delta{a^2}}_{0}  \frac{t^2}{(1+t^2)^{\frac{5}{2}}} \,dt.
 \end{split}    
 \end{equation}
Also in such a case, every function in $O \left(\alpha-\frac{\pi}{2}\right)^2$ gives a negligible term with respect to (\ref{secondo limite}). We obtain an analogous term near $\frac{3\pi}{2}$ and, from the above limit, \eqref{12} follows. \\
\noindent As a consequence, by taking into account \eqref{14}, we have also
 $$\lim_{a\to +\infty} \frac{\mathcal{I}_1 (\sigma,a)}a= -\lim_{a\to +\infty} \frac{\mathcal{I}_1\left(\sigma+\frac{\pi}{2},\frac{1}{a}\right)}a=-c(p(\sigma)+p(\sigma +\pi))$$
 and
 $$\lim_{a\to +\infty} a \, \mathcal{I}_2 (\sigma,a)= -\lim_{a\to +\infty} a \, \mathcal{I}_2\left(\sigma+\frac{\pi}{2},\frac{1}{a}\right)=-d(p^{\prime}(\sigma)+p'(\sigma +\pi))$$
 \indent For every $\sigma\in [0,\frac{\pi}{2}]$, let now $\varepsilon_\sigma :[0,+\infty]\rightarrow\R^2$ be the continuous function defined as follows:
 \begin{equation} \label{15}
     \varepsilon_\sigma (a):=(e^{-|\log(a)|}\mathcal{I}_1(\sigma,a),e^{|\log(a)|}\mathcal{I}_2(\sigma,a))
\end{equation}
where $\varepsilon_\sigma(0)$ and $\varepsilon_\sigma(+\infty)$ are given by \eqref{11} and \eqref{12} respectively. Clearly, \eqref{15} defines a continuous family of curves $(\varepsilon_\sigma)_{\sigma \in [0,\frac{\pi}{2}]}$ in $\R^2$. Moreover, since $p$ is strictly positive, there exists a constant $M > 0$ such that, for every $\sigma \in [0,\frac{\pi}{2}]$ it holds:
\begin{equation} \label{16}
  \begin{split}
 \varepsilon_\sigma(0)\in \{ (x,y)\in \R^2: \ x<-M\} \\
 \varepsilon_\sigma(+\infty)\in \{ (x,y)\in \R^2: \ x>M\}.
  \end{split}  
\end{equation}
Thanks to \eqref{14}, we also note that:
$$\varepsilon_0(a)=-\varepsilon_{\frac{\pi}{2}} (\frac{1}{a}),$$ 
which means that the two curves $\varepsilon_0$ and $\varepsilon_{\frac{\pi}{2}}$ wind up around $(0,0)$. This implies, together with \eqref{16}, that there exists a couple $(a,\sigma)\in \R_+\times[0,\frac \pi 2]$ such that $\varepsilon_\sigma(a)=0$. This concludes the proof.  
\end{proof}

\indent We finally prove how formulas \eqref{tre} and \eqref{3} transform after a change of parameter. 
\begin{lemma}\label{7}
    Let $g:[0,S] \rightarrow [0,2\pi]$ be a bijective diffeomorphism. Moreover, let 
$$L(s_1,s_2)=\omega(\gamma(g(s_1)),\gamma(g(s_2))) \quad \text{and} \quad h(\alpha):= g'(g^{-1}(\alpha))=(g^{-1})'(\alpha).$$
Then we have
\begin{equation}\label{4}
\iint_{[0,S]^2}  \left[L_{11}(s_1,s_2)+L_{22}(s_1,s_2)\right] L_{12}(s_1,s_2) \, ds_1 ds_2=- 2A(D)\int_{[0,2\pi]} (p''+p)^2 h^2\, d\alpha    
\end{equation}
and
\begin{equation}\label{5}
 \begin{split}
&2\iint_{[0,S]^2} L_{12}^2(s_1,s_2)\, ds_1ds_2 \\
&= \left( \int^{2\pi}_0 (p''+p)^2h\, d\alpha \right)^2-\left( \int^{2\pi}_0 (p''+p)^2h\cos(2\alpha)\, d\alpha \right)^2 -\left( \int^{2\pi}_0 (p''+p)^2h\sin(2\alpha)\, d\alpha \right)^2.
    \end{split}   
\end{equation}
\end{lemma}
\begin{proof} By considering \eqref{2} --which holds for every smooth parametrization of $\partial D$-- in the right side we only need to compute
\begin{equation} \label{6}
    \int^S_0 \omega (\gamma (g(s))',\gamma (g(s))'')\,ds. 
\end{equation}
To do that, we just replace 
$$\gamma (g(s))'=\left[p''(g(s))+p(g(s))\right]g'(s)\e_{g(s)}$$
and
\[
\begin{split}
    \gamma (g(s))''=&\left[\left(p'''(g(s)) +p'(g(s))\right)(g'(s))^2 +\left(p''(g(s))+p(g(s))\right)g''(s)\right]\e_{g(s)}\\ &+\left[p''(g(s))+p(g(s))\right](g'(s))^2 J\e_{g(s)}
\end{split}
\]
so that \eqref{6} is equal to
$$ \int^S_0 [p''(g(s))+p(g(s))]^2(g'(s))^3 \, ds$$
and with the change of variable $s=g^{-1} (\alpha)$ we have \eqref{4}. With a similar argument, we obtain \eqref{5}.
\end{proof}
\begin{remark}\label{R1} Since $(p''(\alpha) + p(\alpha))>0$, in Proposition \ref{7} we can choose the usual arc length as parameter   which corresponds to $h(\alpha) = \frac{1}{p''(\alpha) + p(\alpha)}$ and we have
 \begin{equation} \label{REM1}
\iint_{[0,S]^2}  \left[L_{11}(s_1,s_2)+L_{22}(s_1,s_2)\right] L_{12}(s_1,s_2) \, ds_1 ds_2 = - 2A(D)\int_{[0,2\pi]} d\alpha  = - 4\pi A(D)    
\end{equation}   
and 
\begin{equation} \label{REM2}
 \begin{split}
        &2\iint_{[0,S]^2} L_{12}^2(s_1,s_2)\, ds_1ds_2 \\
&= \left( \int^{2\pi}_0 (p''+p)\, d\alpha \right)^2-\left( \int^{2\pi}_0 (p''+p)\cos(2\alpha)\, d\alpha \right)^2 -\left( \int^{2\pi}_0 (p''+p)\sin(2\alpha)\, d\alpha \right)^2.
    \end{split}   
\end{equation}
\end{remark}
\indent We are now ready to prove Theorem \ref{luca}. 
\begin{proof}[Proof of Theorem \ref{luca}]
We first choose an affine transformation $\varphi_{a,\sigma}$ such that Proposition \ref{P1} holds. Then, we parametrize $\partial D_{a,\sigma}$ as in Remark \ref{R1}. With these specific choices of affine transformation and parametrization, and by formulas (\ref{REM1}) and (\ref{REM2}), inequality (\ref{eccola}) of Proposition \ref{DISUG} simply reads:
$$\left( \int^{2\pi}_0 (p_{a,\sigma}''+p_{a,\sigma})\, d\alpha \right)^2 \le 4\pi A(D_{a,\sigma})$$
that is 
$$(l(\partial D_{a,\sigma}))^2 \le 4\pi A(D_{a,\sigma}),$$
where $l(\partial D_{a,\sigma})$ denotes the perimeter length of $\partial D_{a,\sigma}$. Then, by the isoperimetric inequality on the plane, $D_{a,\sigma} = \varphi_{a,\sigma}(D)$ is necessarily a circle. Consequently, by the very construction of $\varphi_{a,\sigma}$, the domain $D$ is an ellipse. 
\end{proof}
\indent We finally remind how Theorem \ref{bialy} follows from Theorem \ref{luca}, see also \cite{B}[Page 154]. 
\begin{proof}[Proof of Theorem \ref{bialy}] It is a standard argument, we sketch the main lines for reader's convenience. We first recall that --since $\Phi$ is a monotone twist map-- a continuous invariant closed curve, which is not null-homotopic, is necessarily a graph of a Lipschitz-continuous function on $\partial D$, see for example \cite{SIB}[Page 9]. Let $\{t_n\}_{n \in \mathbb{Z}}$ be every symplectic billiard configuration whose corresponding phase-space trajectory 
lies on such a curve. Then --since $\Phi$ also preserves an area form-- each finite segment of $\{t_n\}_{n \in \mathbb{Z}}$ is a local maximum, see e.g. \cite{Ban}[Theorem (7.7)]. Moreover (according to \cite{MacKayEtc}[Proposition A1.5.]) the second variation of the functional corresponding to the generating function $L$ is non degenerate. This corresponds to the fact that $\{t_n\}_{n \in \mathbb{Z}}$ has no conjugate points. Therefore, if the phase-space $\mathcal{P}$ is foliated by continuous invariant closed curves not null-homotopic, then the symplectic billiard is without conjugate points. By Theorem \ref{luca}, this means that the billiard table is necessarily an ellipse.
\end{proof}
\bibliographystyle{alpha}
  
 \end{document}